\documentclass[11pt]{amsart}
 \usepackage{amsthm} 
 \usepackage{float}
 \usepackage[bottom=3.5cm, right=2.7cm, left=2.8cm, top=3.0cm]{geometry}

    \usepackage{xcolor}
    \usepackage{hyperref}
         \hypersetup{colorlinks, citecolor=blue, filecolor=blue, linkcolor=black, urlcolor=blue}
    \usepackage{soul}
\usepackage{enumitem}
    \usepackage{mathabx}
    \usepackage[numbers]{natbib}
    \usepackage{mathtools}
    \usepackage{tikz-cd}
    \usepackage{bigints}
    \usepackage{xfrac}
    \usepackage{mathrsfs}
    \usepackage{comment}
    \usepackage{verbatim}
    \usepackage{adjustbox}
    \usepackage{graphicx}
    \usepackage{import}
     \usepackage{upgreek}

\usepackage{marginnote}

\newcommand{\R}{\mathbb{R}}

\newcommand{\T}{\mathbb{T}}\newcommand{\C}{\mathbb{C}}

\numberwithin{equation}{section}

\newtheorem{theorem}{Theorem}[section] 

\newtheorem{lemma}[theorem]{Lemma}
\newtheorem*{lemma*}{Lemma}

\newtheorem*{proposition*}{Proposition}
\newtheorem{conjecture}[theorem]{Conjecture}

\newtheorem*{question*}{Question}
\newtheorem*{theorem*}{Theorem}
\newtheorem*{claim*}{Claim}

\theoremstyle{definition}

\theoremstyle{remark}
\newtheorem{remark}[theorem]{Remark}

\begin{document}

\title{Intermingled basins: Kan's example on the Riemann sphere}

\author{Abbas Fakhari, Ale Jan Homburg and Sebastian van Strien}

\address{A. Fakhari\\ Department of Mathematics,
	Shahid Beheshti University, 19839 Tehran, Iran}
\email{a\_fakhari@sbu.ac.ir}

\address{A.J. Homburg\\ KdV Institute for Mathematics, University of Amsterdam, Science park 107, 1098 XG Amsterdam, Netherlands\newline Mathematical Institute, University of Leiden, PO Box 9512, 2300 RA Leiden, Netherlands \newline  Department of Mathematics,
	Imperial College London,
	180 Queen's Gate,
	London SW7 2AZ,
	United Kingdom}
\email{a.j.homburg@uva.nl}

\address{S.J. van Strien \\  Department of Mathematics,
	Imperial College London,
	180 Queen's Gate,
	London SW7 2AZ,
	United Kingdom}
\email{s.van-strien@imperial.ac.uk}	

\begin{abstract}
Kan's discovery of dynamical systems with two attractors whose basins of attraction both have full support, featured specific examples of skew product systems of interval diffeomorphisms forced by expanding circle maps. The interval diffeomorphisms are polynomial maps and can hence be considered on the Riemann sphere. We consider the resulting skew product systems of holomorphic maps on the Riemann sphere forced by expanding circle maps,
and establish the existence of three attractors whose basins of attraction all have full support and are thus intermingled.
\end{abstract}

\maketitle

\section{Introduction}

Dynamics generated by the iteration of maps can be intricate.  One phenomenon is that of
attractors with riddled basins, which are  basins of attraction of positive measure, but without open sets \cite{MR1206103}.
The notion of attractor used here is the one introduced in \cite{MR0790735}. We refer to such attractors as metric attractors.
Kan \cite{MR1254075} constructed examples of maps on a cylinder with two metric attractors with riddled basins where the support of each basin is the entire cylinder.
That is, each open set intersects both basins in a set of positive measure.
The basins are called intermingled.
Several papers have been written since that develop theory on intermingled basins, we refer to \cite{FAKHARI_HOMBURG_2026} and references therein.

The key example from  \cite{MR1254075} involves a smooth map on the cylinder $\mathbb{T} \times [0,1]$ (where $\mathbb{T} = \mathbb{R}/\mathbb{Z}$), that fixes the boundaries
$\mathbb{T} \times \{0\}$ and $\mathbb{T} \times \{1\}$.
The map is a skew product map with an explicit expression
\begin{align}\label{e:kan}
	(u,x) &\mapsto \left(3 u \pmod 1 , x + \varepsilon \cos(2 \pi u) x (1-x) \right),
\end{align}
with $0 < |\varepsilon| < 1$.
The boundaries $\mathbb{T} \times \{0\}$ and $\mathbb{T} \times \{1\}$ in this example are metric attractors with intermingled basins.
In more detail, for Lebesgue almost any $u \in \mathbb{T}$, there is $0 < r(u) < 1$, so that $\{u\} \times [0,r(u))$ is in the basin of  $\mathbb{T} \times \{0\}$, and $\{u\} \times (r(u),1]$ is in the basin of  $\mathbb{T} \times \{1\}$.
The values of $r$ vary wildly so that every open set intersects both basins in sets of positive measure.
This description implies that the boundary of the cylinder is the 
 likely limit set, that is, the omega-limit set of a full measure set of initial points lies in $\cup_{p=0,1} \mathbb{T}\times \{p\}$.
A picture of the basins can be found in \cite{MR2477416}.

Noting that \eqref{e:kan} involves an expression that is polynomial in $x$, a natural choice is to take
$x$ from the complex numbers. That is the focus of this paper: we explore the existence of metric attractors and their basins in Kan's example, but considered with fiber maps on the Riemann sphere $\hat{\mathbb{C}}$.
 For $\alpha \in \mathbb{C}$ and a parameter $u \in \mathbb{T}$,
consider the holomorphic map $f_u : \hat{\mathbb{C}} \righttoleftarrow$ given by
	\begin{align}\label{e:fu}
	f_u(z) &:= z + \alpha \cos(2 \pi u) z (1-z).
	\end{align}
Note that the points $0, 1, \infty$ are fixed points of $f_u$ for any $u \in \mathbb{T}$.
We consider such maps that are forced, through dynamics on $u$, by sufficiently expanding circle maps. 	
For a natural number $L > 1$, take the expanding circle map $E$ on  $\mathbb{T}$ given by
\[
E (u) := L u \pmod 1,
\]
and consider
the skew product map
$F: \mathbb{T} \times \hat{\mathbb{C}} \righttoleftarrow$  defined by
\begin{align}\label{e:skew}
F (u,z) &\coloneqq \left(E (u)  , f_u(z) \right) = \left( L u \pmod 1 , z + \alpha \cos(2 \pi u) z (1-z) \right).
\end{align}
For positive iterates of $F$ we use notation
\begin{align*}
	F^n (u,z) &= \left(E^n (u)  , f^n_u(z) \right),
\end{align*}
where $f^n_u (z) := f^n_{E^{n-1} u} \circ \cdots \circ f_u (z)$.

The skew product map $F$ is equivariant with respect to the involution  $J : \mathbb{T} \times \hat{\mathbb{C}} \righttoleftarrow$ given by
  \[
  J(u,z) = \Big( u + \frac{1}{2} \pmod 1 , 1-z \Big).  
  \]
  That is, $F \circ J = J \circ F$.
  In particular $f_u$ and $f_{u + 1/2}$ are conjugate using $z \mapsto 1-z$.
  
\begin{remark}\label{r:equivariant}
	At $u=1/4$ and $u=3/4$ we find $f_u (z) \equiv z$. In particular we find that the degree of $f_u$ is not constant in $u$, and that the family of maps $u \mapsto f_u$
    does not give a continuous family of rational maps with fixed degree on the Riemann sphere. Moreover, the circle $\mathbb{T} \times \{\infty\}$ cannot be an attractor with an open basin of attraction. The inverse images of $0\in\hat{\mathbb{C}}$ (equal to $(1+\alpha \cos(2 \pi u))/(\alpha \cos (2 \pi u))$ for $u \neq 1/4, 3/4$) and of $1$ (equal to $1/(\alpha \cos(2 \pi u))$ for $u \neq 1/4,3/4$) approach $\infty$ for $u \to 1/4,3/4$.   
  For $L$ odd, $u = 1/4$ and $u = 3/4$ in $\mathbb{T}$ are either fixed points or period-two points for $E$.  
  For $L$ even, $1/4$ and $3/4$ are eventually mapped to $0$ by iterates of $E$.
\end{remark}	

Our main result is Theorem~\ref{t:main} below.
In it we prove that for suitable $\alpha$ with $0 < |\alpha| < 1$ and under mild non-resonance conditions on eigenvalues $f'_u (0)$ and $f'_u (1)$ for chosen fixed fibers, the skew product map $F$ from \eqref{e:skew} has metric attractors
$\mathbb{T} \times \{0\}$, $\mathbb{T} \times \{1\}$,
$\mathbb{T} \times \{\infty\}$ with mutually intermingled basins; all three basins have support equal to the entire state space $\mathbb{T} \times \hat{\mathbb{C}}$. 
The non-resonance conditions allow for rational real $\alpha \in (-1,0) \cup (0,1)$ (as used in numerical experiments with Kan's example) with  $L$ odd.
We do not have a proof that the union of the three circles is the likely limit set, see the discussion in Section~\ref{s:lls}.

Let us mention some contexts and studies that have a relation to this paper.
The possible occurrence of intermingled basins for Desboves maps on the real projective plane is discussed
in \cite[Section~6]{MR2378483}; these maps admit two invariant curves that are metric attractors whose basins appear to be intermingled.
The paper \cite{MR1406437} has examples of polynomial maps on $[0,1]$ with Cantor sets as metric attractors that have intermingled basins, and further
discusses corresponding constructions of rational maps on the Riemann sphere, relying on the existence of a Julia set of positive Lebesgue measure.

Starting with papers such as \cite{MR1429333,MR1619744,MR1488318} that explored basic properties, a large body of work on nonautonomous and random holomorphic dynamics has appeared, see for instance \cite{MR3084426,MR2747724}. This includes work on random quadratic maps such as  \cite{MR1619744,MR4403145,MR4803667,MR1721617}. Skew product maps of (real) quadratic maps over expanding circle maps (Viana maps) have also been considered for instance in \cite{MR2018605,MR3007895,MR1471866},
typically with a focus on nonuniform expansion.
%

Families of rational maps with parameter independent fixed points also arise in relaxed Newton maps $z \mapsto z + a P(z) / P'(z)$. Randomized versions occur in contexts of random relaxed Newton methods \cite{MR4268827}.
The paper \cite{MR2871980} investigates a family of rational maps depending on a parameter $a$ with parabolic fixed points at $0$ and $\infty$, for all values of $a$.
Replacing the parameter $a$ by an expression $\alpha \cos (2 \pi u)$ and having $u$ varied through an expanding circle map would give a setting akin to our paper.
	

%

\section{Kan's example on the Riemann sphere}

Consider the skew product system  \eqref{e:skew} of holomorphic maps forced by an expanding circle map. Note that $Df_u (0) = 1 + \alpha \cos (2 \pi u)$ and 
$Df_u(1) = 1 - \alpha \cos (2 \pi u)$.
We assume 
\begin{equation}\label{l:0alpha1}
	0 < |\alpha| < 1,
\end{equation}
so that $Df_u(0)$ and $Df_u(1)$ are not $0$
for all $u \in \mathbb{T}$,  and we assume negative fiber Lyapunov exponents at $p = 0,1$:
\begin{align}\label{eq:negative-lyap-u-at0}
	\lambda(p) \coloneqq \int_{\mathbb{T}} \ln \left| Df_{u}  (p) \right| \, du &=  \int_{\mathbb{T}}  \ln \left| 1 + \alpha \cos(2 \pi u)  \right| \, du  =  \ln \left|  \frac{1 + \sqrt{1-\alpha^2}}{2}  \right| < 0
\end{align}
(see Lemma~\ref{l:integral} in Appendix~\ref{s:integral}), where we note that $\lambda(1)$ equals $\lambda(0)$.



\begin{remark}
	For real $\alpha$ we have that \eqref{eq:negative-lyap-u-at0} holds with  $0<|\alpha| < 1$. 
Take $L$ to be odd, so that $0$ and $1/2$ are fixed points of $E$.
For real $\alpha \in (0,1)$ we find that $0$ is attracting for $f_{1/2}$ while
$1$ is attracting for $f_{0}$. For real $\alpha \in (-1,0)$ this is the other way around.  In general the assumptions \eqref{l:0alpha1} and \eqref{eq:negative-lyap-u-at0} on $\alpha$ imply that $0$ is a hyperbolic attracting fixed point for either $f_0$ or $f_{1/2}$, leaving $1$ to be a hyperbolic attracting fixed point for the other fiber, see Lemma~\ref{l:alpha} in Appendix~\ref{s:integral}. 
\end{remark}	

Having two fixed fibers for which $0$ is attracting in one fiber and $1$ is attracting in the second  fiber, is pivotal in our analysis.
We use general notation and assume fixed points $u_0, u_1 \in \mathbb{T}$ of $E$ with
\begin{equation}\label{eq:u0}
	\ln \left| Df_{u_0}  (0) \right|  < 0, \ln \left| Df_{u_0}  (1)   \right| > 0 \text{ and } 
	\ln \left| Df_{u_1}  (0) \right|  > 0, \ln \left| Df_{u_1}  (1)   \right| < 0.
\end{equation}

For $p$ equal to $0$, $1$ or $\infty$ in the Riemann sphere,
write
\[
W^s(p) \coloneqq \left\{  (u,z) \in  \mathbb{T} \times \hat{\mathbb{C}} \; : \; \lim_{n\to \infty}  f^n_u (z)  = p \right\}
\]
for the basin of attraction of $\mathbb{T} \times \{p\}$.
For  $u \in \mathbb{T}$, define the slice
\[
W^s_u(p) \coloneqq \left\{  (u,z) \in   \{u\} \times \hat{\mathbb{C}} \; ; \; \lim_{n\to \infty}  f^n_u (z)  = p \right\}.
\]
For $p = 0,1,\infty$, write $B^s (p)$ for the support of induced volume on  $W^s(p)$
(the closed set in $\mathbb{T} \times \hat{\mathbb{C}}$ of all points for which any open neighborhood of them contains positive measure of $W^s(p)$).

\begin{remark}\label{r:periodic}	
	Dynamics varies wildly in the fibers. 	
	Periodic fibers and eventually periodic fibers are dense and may show dynamics different from the description in Theorem~\ref{t:main} for typical fibers. 
	For instance for $p = u_0, u_1$, 
	\begin{equation}\label{e:dense}
	\left\{   u \in \mathbb{T} \; ; \;  E^i (u) = p \text{ for some } i \in \mathbb{N} \right\} \text{ is dense in } \mathbb{T}.
	\end{equation}
	The only attracting fixed points of $f^n_{\tilde{u}}$ for $\tilde{u}$ that is eventually mapped to $u_0$
	are $0$ and $\infty$, while the  only attracting fixed points of $f^n_{\tilde{u}}$ for $\tilde{u}$ that is eventually mapped to $u_1$
	are $1$ and $\infty$. 
	
	Additional attractors may occur in some fibers.
	Take for instance $\alpha=0.3$, $L=3$ (values that are also taken for the numerically calculated figures below)   and $\tilde{u} = 1/5$. Then $E^4 (\tilde{u}) = \tilde{u}$. A graphical analysis of $f^4_{\tilde{u}}$ on the real axis shows that $f^4_{\tilde{u}}$ has an attracting fixed point near $4.6015$. Also $0$ is attracting for $f^4_{\tilde{u}}$, but $1$ is a repelling fixed point. 
\end{remark}	

The following lemma shows that  $\mathbb{T} \times \{p\}$, $p = 0,1,\infty$, are metric attractors, and that
 in small tubular neighborhoods of them, large proportions of points are in the basins of attraction. Write $B_R(p)$ for the disc of radius  $R$ in the Riemann sphere around $p$.

\begin{lemma}\label{l:Ws-loc}
	Let $p \in \hat{\mathbb{C}}$ be $0$, $1$, or $\infty$.
For any $\varepsilon>0$ there is $R>0$ so that there is a set $\Lambda_p \subset \mathbb{T}$ with Lebesgue measure at least $1 - \varepsilon$, so that for $u \in \Lambda_p$,
$\{u\}\times B_R(p) \subset W^s_u (p)$.
\end{lemma}			
		
\begin{proof}
For $p=0$ and $p=1$, the proof follows  \cite[Lemma~2.2]{MR1254075}.

Near $p=\infty$ one treats the fiber maps $f_u$, as usual, in a coordinate $w = 1/z$.
Written in this coordinate $z_{n+1} = f_u(z_n)$ becomes
\begin{align*}  
	w_{n+1}
	&=  g_u (w_n) \coloneqq \frac{w_n^2}{ (1 + \alpha \cos (2 \pi u)) w_n - \alpha \cos (2 \pi u)   }.
\end{align*}
Note that $g_u'(0) = 0$ for $u \in \mathbb{T} \setminus \{ 1/4,3/4\}$ and $g'_{1/4} (0) = g'_{3/4} (0) = 1$.
When $u$ is near $1/4$,
write $g_u(w) = w^2 /  (( 1+ \alpha s ) w - \alpha s)$ for $s = \cos (2 \pi u)$ near $0$.
Note that
\[
g_u' (w) = \frac{w^2 (1 + \alpha s) - 2 \alpha s w}{ (w (1 + \alpha s) - \alpha s)^2 }.
\]
For $w = k s$ we find
\[
g_u'(w) = \frac{ k^2 (1 + \alpha s) - 2 \alpha k }{  (k (1 + \alpha s) -\alpha)^2}.
\]
Because $\alpha$ is fixed, we find that $|g_u'(w)|$ is small for $s, |k|$ near $0$.

A similar analysis applies near $u = 3/4$.
To conclude,
there are two cones
$C_{1/4,K} = \{ (u,w) \in \mathbb{T} \times \mathbb{C} \; ; \; |w| \leq \varepsilon, |w| \leq K |u - 1/4| \}$ and $C_{3/4,K}$ defined analogously, so that on
\[
V \coloneqq \{  (u,w) \in \mathbb{T} \times \mathbb{C} \; ; \;  |w| \leq \varepsilon \} \setminus \left(  C_{1/4,K} \cup  C_{3/4,K} \right),
\]
we have $|g_u'(w)| \leq \lambda$ for some $\lambda <1$.
Consider sets
\[
I_{K,n} := \left[\frac{1}{4} - \frac{\varepsilon}{K} \lambda^n , \frac{1}{4} + \frac{\varepsilon}{K}  \lambda^n  \right] \cup   \left[\frac{3}{4} - \frac{\varepsilon}{K} \lambda^n  , \frac{3}{4} + \frac{\varepsilon}{K} \lambda^n  \right].
\]
Put $G(u,w)\coloneqq (E(u),g_u(w))$. An orbit $(u_n,w_n) = G^n (u_0, w_0)$ with $|w_0| < \varepsilon$ and $u_n \not \in I_{K,n}$ for all $n \geq 0$, is an orbit contained in $V$. 
Note that $u_n \in I_{K,n}$ is equivalent to  $u_0 \in E^{-n} (I_{K,n})$, which is a set of Lebesgue measure $|E^{-n} (I_{K,n})| = |I_{K,n}| = 2 \varepsilon \lambda^n / K$. Since $\sum_{n\geq 0} |I_{K,n}| < \infty$, the Borel-Cantelli lemma gives that for Lebesgue almost all $u_0 \in \mathbb{T}$, $u_n = E^n(u_0) \in I_{K,n}$ occurs at most finitely often.
It follows  that $\Lambda_N=\{u\in \T; E^n(u)\in I_{K,n} \mbox{ for all }n\ge N\}$ 
has positive Lebesgue measure for all $N$ and $|\Lambda_N|\to 1$ as $N\to \infty$. 
For each $N$  there is a positive $d_0>0$ and so that for each $u_0\in \Lambda_N$ and each $w_0$ with $|w_0|<d_0$, 
$(u_n,w_n)=G^n(u_0,w_0)\in V$ for all $n$. 
As $|g_u' (w)| \leq \lambda$ whenever $(u,w) \in V$,
for such initial points $(u_0,w_0)$,  $(u_n,w_n)$ converges to $\mathbb{T} \times \{0\}$ as $n \to \infty$.
\end{proof}			

Taking the union over small positive $\varepsilon$ of the discs $\{u\} \times B_R (p)$ constructed in the above lemma yields a local stable set $W^s_{loc} (p)$.  
The stable sets $W^s (p)$ equal the union of inverse images
	$\cup_{n \in \mathbb{N}} F^{-n} ( W^s_{loc} (p))$, modulo a set of zero measure.

Fatou sets $\mathcal{F}_u$ and Julia sets $\mathcal{J}_u$ are defined as subsets of $\{u\} \times \hat{\mathbb{C}}$ in the usual way \cite{MR1429333,MR1619744}. That is,
\[
\mathcal{F}_u = \left\{  (u,z) \in \{u\} \times \hat{\mathbb{C}} \; ; \;  \left\{ f_u^n \right\} \text{ is normal on some neighborhood of } (u,z) \right\}
\]
and
\[
\mathcal{J}_u = ( \{u\} \times \hat{\mathbb{C}} ) \setminus \mathcal{F}_u.
\]
The Julia set $\mathcal{J}_{u_1}$  is a Jordan curve, see \cite[\S~9.9]{MR1128089}, encircling $1$ and  containing  $0$.
Likewise  $\mathcal{J}_{u_0}$  is a Jordan curve, encircling $0$ and containing $1$.
\begin{figure}[hb]
	
	\includegraphics[width=7cm,height=7cm]{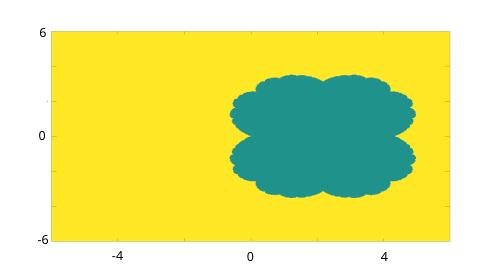} \includegraphics[width=7cm,height=7cm]{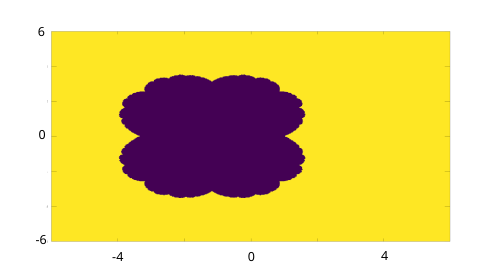}

	\caption{Pictures of filled Julia sets for $\alpha = 0.3$ and $L=3$. The sets $W^s_0(1)$ (light blue region in the left picture) and $W^s_{1/2}(0)$ (dark blue region in the right picture) are bounded by the Julia sets $\mathcal{J}_{0}$ (encircling $1$, containing $0$)  and $\mathcal{J}_{1/2}$ (encircling $0$, containing $1$) respectively.}\label{f:julia}
\end{figure}	
Figure~\ref{f:julia} depicts the filled Julia sets (the closures of $W^s_u(0)$ and $W^s_u(1)$) for the fixed fibers $u=0$ and $u=1/2$ with $\alpha = 0.3$ (recall from Remark~\ref{r:equivariant} that $f_0$ and $f_{1/2}$ are conjugate).
 Figure~\ref{f:1oversqrt3and5} illustrates $W^s_u(0)$ and $W^s_u(1)$ for
 two different irrational values of $u$, also with $\alpha = 0.3$.
\begin{figure}[hbt]
	
	\includegraphics[width=7cm,height=7cm]{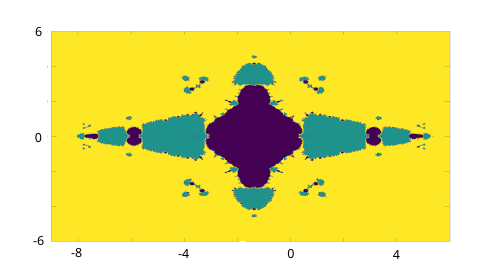} \includegraphics[width=7cm,height=7cm]{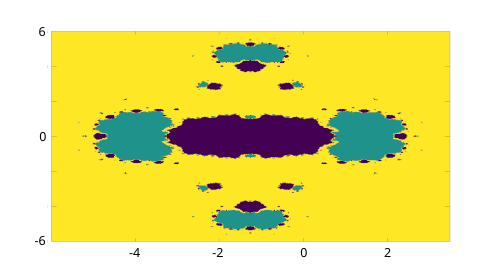}	
	
	\caption{Numerically calculated sets $W^s_u(0)$, $W^s_u(1)$ for two irrational values of $u$, with  $\alpha = 0.3$, $L=3$.
		For the left picture we chose
		$u = 1/\sqrt{3}$, and $u = 1/\sqrt{5}$  for the right picture. Colors are as in Figure~\ref{f:julia}: dark blue for $W^s_u(0)$ and light blue for $W^s_u(1)$. \label{f:1oversqrt3and5}  }
	
\end{figure}

%
%

To make the construction below feasible, we will apart from \eqref{eq:u0} consider conditions on multiplicative independence of the moduli of $Df_{u_0}$ and $Df_{u_1}$ at the fixed points $p = 0,1$:
\begin{align}\label{eq:rat-indep-at0}
	\ln |Df_{u_0} (p)| / \ln|Df_{u_1} (p)| &\not\in \mathbb{Q},
\end{align}	
For $p=0,1$, consider
\[
S_p \coloneqq  \overline{  \left\{ \left(Df_{u_0} (p)\right)^k \cdot \left(Df_{u_1}(p)\right)^\ell \; ; \; k, \ell \in \mathbb{N} \right\} }.
\]
If $\alpha$ is a real number, then \eqref{eq:rat-indep-at0} (assuming also \eqref{eq:u0}) implies that $S_p = [0,\infty)$.

\begin{remark}
Consider for $L$ odd, the fixed points $u_0 = 1/2$ and $u_1 = 0$.	
Then $Df_{u_0}(0) = 1 - \alpha$ and $Df_{u_1} (0)  = 1 + \alpha$.
Now suppose that $\alpha \in \mathbb{Q} \cap (0,1)$, say $\alpha = \frac{a}{b}$ with $b>a >0$ and  $\gcd (a,b)  =1$.
The following reasoning shows that $\ln(1+\alpha)$ and $\ln(1-\alpha)$ are rationally independent.

If not, then we have an equation
\[
m \ln (1 + \alpha) + n \ln (1-\alpha) = 0
\]
for nonnegative integers $m,n$ that are not both zero.
Hence
\[
\big( \frac{b+a}{b}  \big)^m = \big( \frac{b}{b-a}  \big)^n,
\]
or
\begin{equation} \label{e:prime}
(b+a)^m (b-a)^n = b^{n+m}.
\end{equation}
Let $d = \gcd (b+a, b-a)$ and suppose the prime number $p$ divides $d$.
Then $p$ also divides $(b+a) + (b-a) = 2b$ and $(b+a) - (b-a) = 2a$. As $\gcd (a,b) = 1$ the only common prime factor of $b+a$ and $b-a$ is $2$.
From \eqref{e:prime} we find that $a=0$. This was excluded.
So,  \eqref{eq:rat-indep-at0}  holds for the case $L$ odd, $u_0=1/2$, $u_1 = 0$, and $\alpha$ rational  in $(0,1)$.
\end{remark}

Recall that real numbers $\{a,b,c\}$ are called incommensurable if their quotients
$a/b, a/c, b/c$ are irrational.
For non-real derivatives at the fixed points we  assume in addition to the conditions on the moduli of the derivatives the following conditions on their arguments: for $p=0,1$, 
\begin{align}\label{e:argat0}
	\left\{ \arg (Df_{u_0} (p)) ,  \arg (Df_{u_1} (p)) , 2 \pi \right\} & \text{ are incommensurable}. 
\end{align}
%
%
%
%
%
%
%
Conditions \eqref{eq:rat-indep-at0} and \eqref{e:argat0} together (again with \eqref{eq:u0}) imply that $S_p = \mathbb{C}$, $p=0,1$.

The following lemma gives local linearizing coordinates near $(\hat{u},0)$ with $\hat{u}$ a fixed point of $E$. We assume $\hat{u} \neq 1/4,3/4$ to avoid derivatives $Df_{\hat{u}} (0)$ equal to $1$. To avoid excessive notation we only consider the fixed point $(\hat{u},0)$, but completely analogous results hold near $(\hat{u},1)$.
Write
\[
\eta(u) := Df_u (0).
\]
The lemma says that near $(\hat{u},0)$ we may find a coordinate change
turning the skew product system into
\[
(u,z) \mapsto (E(u) , \eta(\hat{u}) z).
\]
%
%
The possibility  of an analytic linearization is covered by Siegel's theorem, see \cite{MR947141}. It requires Diophantine conditions on eigenvalues and is not always possible in our set-up.  We therefore
use complex coordinates $(u,z,\bar{z})$ and consider $C^k$ for high $k$, but possibly non-analytic, coordinate transformations $H (u,z,\bar{z}) = (u , H_u(z,\bar{z}))$.


\begin{lemma}\label{l:lin}
		Let $\hat{u}$ be a fixed point of $E$, different from $1/4$ and $3/4$.
	For any positive integer $k$ one has that for large enough $L$, the following holds:
	
	\begin{enumerate}
	\item
	 There are $C^k$ linearizing coordinates $(u,z,\bar{z}) \mapsto H(u, z,\bar{z}) := (u , H_u(z,\bar{z}))$  near the fixed point $(\hat{u} , 0)$.
	Here $H_u(0,0)=0$, $D_z H_u(0,0)= 1$, $D_{\bar{z}} H_u(0,0) = 0$ and
	\[
	H_{E(u)}^{-1}  \big( \eta(\hat{u})   H_u (z,\bar{z}) \big) = f_u(z).
	\]	
	Likewise there are $C^k$ linearizing coordinates near the fixed point $(\hat{u} , 1)$;

 \item
	There are $C^k$ coordinates $(u,z,\bar{z}) \mapsto H(u, z,\bar{z}) = (u , H_u(z,\bar{z}))$ near the fixed point $(\hat{u} , \infty)$, in which the map $F$ becomes $(u,w) \mapsto (E(u), w^2)$.
\end{enumerate}
\end{lemma}	

\begin{proof}
	It is a classical result (K\oe{}nigs linearization
	theorem, see for instance \cite[Theorem~6.3.3]{MR1128089} or \cite[Theorem~6.1]{MR2193309}) that the holomorphic map $f_{\hat{u}}$ near the fixed point $0$ can be locally linearized by an analytic coordinate change.
	We thus have a locally analytic map $h_{\hat{u}}$ with $h_{\hat{u}}(0)=0$, $D h_{\hat{u}}(0)= 1$ so that
	\[
	h_{\hat{u}}^{-1}  \big( \eta(\hat{u})   h_{\hat{u}} (z) \big) = f_{\hat{u}}(z).
	\]	
			
The linearizing coordinates can be extended to values of $u$ near $\hat{u}$ through the  construction of a strong unstable foliation for $F$ near $(\hat{u},0)$. Such a strong unstable foliation is $C^k$ for $L$ high by spectral gap conditions.
(The foliation is obtained from a section of a Grassmannian bundle, this section defines a normally hyperbolic manifold for the lifted system whose smoothness is given by spectral gap conditions, see also \cite{MR501173}.)
The foliation defines  $C^k$ holonomy maps $I_{u,\hat{u}} : \{u\} \times V \to \{\hat{u} \} \times \mathbb{C}$ defined for $u$ near $\hat{u}$ and $V$ a small open neighborhood of $0$.
%
%
With a slight abuse of notation,
the coordinate transformation is given by
\[
H_u (z , \bar{z}) \coloneqq h_{\hat{u}} \circ I_{u,\hat{u}} (z, \bar{z}).
\]

%
%

The proof for coordinates near  $(\hat{u} , \infty)$ copies the argument, replacing K\oe{}nigs linearization
theorem for a local normal form near a hyperbolic point by B\"ottcher's theorem, see \cite[Theorem~6.7]{MR2193309}.
\end{proof}


\begin{remark}
An alternative normal form near $(\hat{u},0)$
turns the skew product system into
\[
(u,z) \mapsto (E(u) , \eta(u) z).
\]
%
In more detail, for $u$ near $\hat{u}$ and $z$ near $0$ there is a $u$-dependent  coordinate $h_u(z,\bar{z})$, $h_u (0,0)=0$, $D_zh_u(0,0) = 1$, $D_{\bar{z}} h_u (0,0) = 0$ so that
\[
 h_{E(u)}^{-1}  \big( \eta(u)   h_u (z,\bar{z})  \big) = f_u(z).
\]
This may be compared to a global (in $u$) linearization as in  \cite{MR2158403}, that depends measurably on $u$. See also \cite{MR3517619}.
%
%
%

To obtain this normal form, we must find a conjugation between
the skew products $(u,z) \mapsto (E(u) , \eta (\hat{u}) z)$ and $(u,z) \mapsto  (E(u) , \eta(u) z)$ near $(\hat{u},0)$.
To see this we use a coordinate change $I(u,z) = (u , i_u z)$ (so linear but $u$-dependent in the $z$-coordinate). A conjugation means
\[
 \eta (0) z = \frac{1}{i_{E(u)}} \eta (u) i_u z,
\]
and thus
	\[
	i_{E(u)} = \frac{\eta(u)}{\eta(u_0)}  i_u.
	\]
Consider the system $(u,i) \mapsto \left(E(u) ,  \frac{\eta(u)}{\eta(u_0)}  i \right)$. The sought for conjugation is determined by the strong unstable manifold through the point $(0,1)$.
%
\end{remark}

\begin{lemma}\label{l:cones}
	Let $p \in \{0,1,\infty\}$ and $(u,z) \in B^s(p)$. Let $O$ be a neighborhood of $(u,z)$.

	Assume $p \in \{0,1\}$. Then for each $\delta >0$ there is $N>0$ so that $F^N(O)$ contains a curve $\gamma$ that is $\delta$-$C^1$ close to $\mathbb{T} \times \{p\}$.

	Assume $p=\infty$ and let $J \subset \mathbb{T}$ be a closed interval disjoint from 
	$\{1/4,3/4\}$. Then for each $\delta >0$ there is $N>0$ so that $F^N(O)$ contains a curve $\gamma$ that is $\delta$-$C^1$ close to $J \times \{p\}$.
\end{lemma}		

\begin{proof}
	Assume first $p \in \{0,1\}$.
	The reasoning is based on constructions with invariant cone fields, classical in dynamical systems theory \cite{MR1326374}.
	For $(u,z) \in \mathbb{T}\times \mathbb{C}$ and $a>0$ we consider the cone $C_a \subset T_{(u,z)} \left(\mathbb{T}\times \mathbb{C}\right)$ with diameter $a$: identifying
	 $T_{(u,z)} \left(\mathbb{T}\times \mathbb{C}\right)$ with $\mathbb{R}\times \mathbb{C}$, 
	 $C_a$ is given by \[C_a \coloneqq \{ (v,w) 	\in \mathbb{R} \times \mathbb{C} \; ; \; |w| \leq a |v| \}.\] 
	Write $B$ for a small ball around $p$ in $\hat{\mathbb{C}}$. 
	On $B$, $|Df_u |$ is bounded, uniformly in $u$. Hence,
	for $L$ large enough, $DF (u,z)$ with $(u,z) \in \mathbb{T} \times B$ maps $C_a$ inside itself. 
	
	For $n$ large there are intervals $I_n$ containing $u$ so that $I_n \times \{z\} \subset O$ and $E^{n} (I_n)$ has length  one.  
	For $m$ large enough we find $f^m_u(z) \in B$. 
	There is $a>0$ so that for any large $n$, $DF^n (u,z) (1,0) \subset C_a$ for any $u \in I_n$. Lemma~\ref{l:lin}, noting that $f_u(p)=p$ for all $u \in \mathbb{T}$,  makes clear that $F^n (I_n \times \{z\})$ is a curve
	accumulating onto $\mathbb{T} \times \{p\}$ with $DF^n (u,z) (1,0)$ going to zero for $u \in I_n$ as $n\to \infty$.
	This implies the lemma for $p \in \{0,1\}$.

	
	For $p = \infty$ the reasoning is similar.  Introduce a coordinate $w = 1/z$ as in the proof of Lemma~\ref{l:Ws-loc}.
	Let $I_{K,n}$ be as in the proof of Lemma~\ref{l:Ws-loc}. 
	Because $(u,z) \in B^s(\infty)$, we may, by iterating, 
	assume that $z$ is close to $\infty$, that is, $w = 1/z$ is close to $0$.
	We may moreover replace $(u,w)$ by a point in $W^s (\infty)$, which we again denote by $(u,w)$, so that
	$E^n (u) \in I_{K/2,n}$ at most finitely often. By iterating we can then also assume that  $E^n (u) \not \in I_{K/2,n}$ for all $n \geq 0$.
%
	At points that are close to $\mathbb{T} \times \{\infty\}$ but outside of the set $C_{1/4,K} \cup C_{3/4,K}$, there are invariant cones $C_a$ as above. One can therefore iterate a curve near $(u,w)$ with tangent lines in the cones $C_a$: as long as the image is outside $C_{1/4,K} \cup C_{3/4,K}$ this maps to a curve with tangent lines in the cones $C_a$. Take such a curve given as a graph defined over a small interval $J \subset \mathbb{T}$ containing $u$.
	Let $J_n \subset E^n(J)$ be the connected component of $E(J_{n-1}) \cap \left( \mathbb{T} \setminus I_{K,n} \right)$ that contains $u_n$. We claim that $\liminf_{n\to \infty} |J_n| \geq c$ for some $c>0$.
	Note that $|J_n| = L |J_{n-1}|$ unless $E(J_{n-1})$ intersects $I_{K,n}$. In that case, as $u_n$ is outside $I_{K/2,n}$  we have $|J_n| \geq d \lambda^n$ for some $d>0$. 
	Further iterates increase the length, and since $E$ is expanding linearly, we have that $|J_m|$ is uniformly bounded away from zero for some $m > n$, and
	$I_j$ does not intersect $I_{K,j}$ for $n < j \leq m$.
	This implies the claim. Recall that inverse images of fixed points of $E$ are dense in $\mathbb{T}$ (see \eqref{e:dense}).
	We may thus take $J$ so that $J_m$ contains a fixed point other than $1/4$ or $3/4$. The lemma for $p=\infty$ follows as above.
\end{proof}	


The following result is the main result of this paper. It states that the basins of attraction of $\mathbb{T} \times \{0\}$, $\mathbb{T} \times \{1\}$ and $\mathbb{T} \times \{\infty\}$ are all dense in the entire state space $\mathbb{T} \times \hat{\mathbb{C}}$. This means that the basins are intermingled: every open set intersects the supports of all basins. 
In addition we prove that $F$ is topologically mixing, so in particular topologically transitive. This means that from a topological point of view the dynamics differs from the above description: there is a residual set of points with dense orbits \cite[Section~18]{MR584443}.
Note that the theorem does not include a statement that the three attractors form the likely limit set.

\begin{theorem}\label{t:main}
Consider the skew product system $F: \mathbb{T} \times \hat{\mathbb{C}}\righttoleftarrow $ given by \eqref{e:skew}, with $f_u : \hat{\mathbb{C}} \righttoleftarrow$ given by \eqref{e:fu}. Assume $0 < |\alpha | < 1$ and \eqref{eq:negative-lyap-u-at0}. 
	
	Let also $u_0 = E(u_0)$ and $u_1 = E(u_1)$ be as above. 
	For $\alpha$  real, assume \eqref{eq:u0}, 
    \eqref{eq:rat-indep-at0}. 
	For $\alpha$ non-real, assume in addition \eqref{e:argat0}. 
	
	Then for large enough and odd $L$,  the following holds:
	\begin{enumerate}
		
	\item	
	For $p = 0,1,\infty$,
	\[
	B^s(p) = \mathbb{T} \times \hat{\mathbb{C}},
	\]
	so that $\mathbb{T} \times \{0\}$, $\mathbb{T} \times \{1\}$ and $\mathbb{T} \times \{\infty\}$ are {metric} attractors
	with intermingled basins;


    \item
	$F$ is topologically mixing.

	\end{enumerate}
\end{theorem}
	
A large part of the argument to prove the theorem is contained in the following inclusion lemma.

\begin{lemma}
        $B^s(p)\subset B^s(q)$, for $p,q=0,1,\infty$.
\end{lemma}  

\begin{proof} We prove three cases, $B^s(0) \subset B^s(1)$, $B^s(0) \subset B^s(\infty)$ and $B^s(\infty) \subset B^s(0)$. The remaining three cases are entirely analogous. 
\begin{figure}[ht]
\includegraphics[width=15cm]{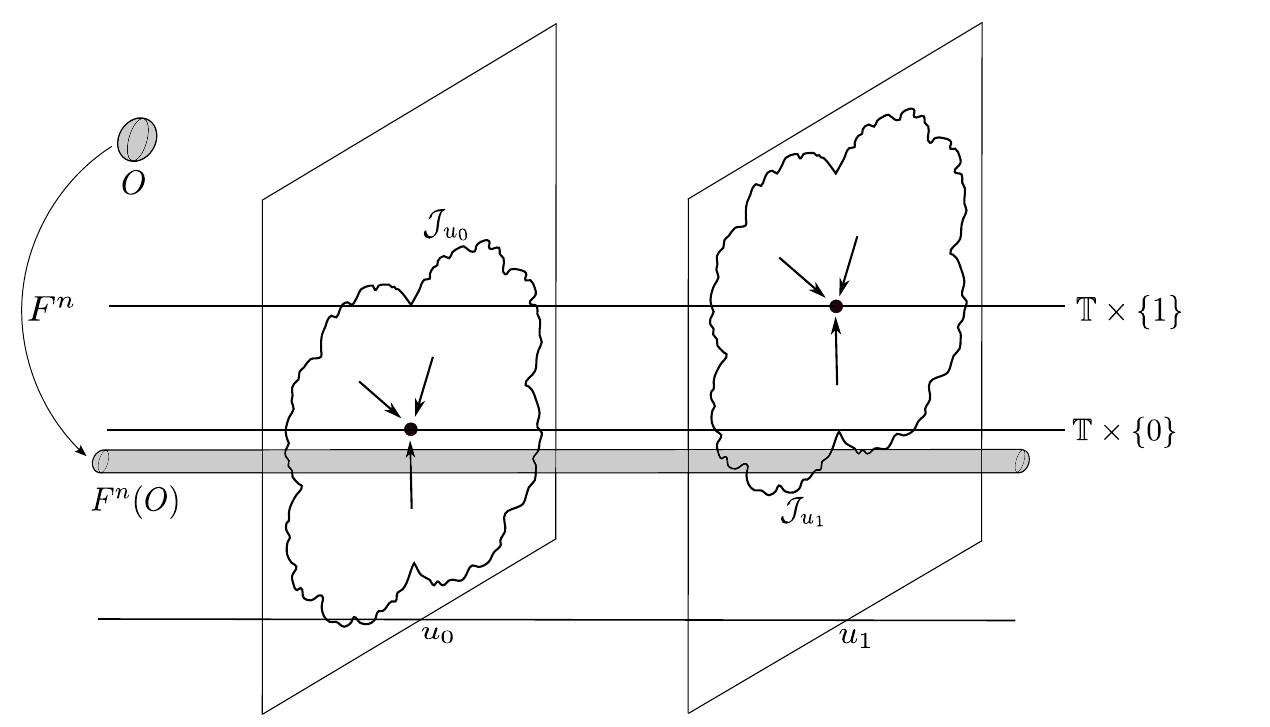}
		
		\caption{\label{f:proof-sketch} Iterates of an open set $O$ that intersect the basin  of attraction of $\mathbb{T} \times \{0\}$ contain parts that accumulate onto $\mathbb{T}\times \{0\}$. To prove that the basins of $\mathbb{T} \times \{0\}$ and $\mathbb{T} \times \{1\}$ are intermingled, we show that  further iterates contain parts that accumulate onto $\mathbb{T} \times  \{1\}$.   }
	\end{figure}	
    
	\begin{proof}[Proof of $B^s(0) \subset B^s(1)$]\renewcommand{\qedsymbol}{}
    Take a point in $B^s(0)$ and a neighborhood $O$ of it. Then $O$ intersects $W^s(0)$ in a set of positive measure. 	
	We must show that $O$ also  intersects $W^s(1)$ in a set of positive measure.
	We do this by showing that parts of $F^n(O)$ accumulate on $\mathbb{T} \times \{1\}$.
	The argument will be finished by invoking Lemma~\ref{l:Ws-loc}.
	
	Since $O$ intersects $B^s(0)$, using Lemma~\ref{l:cones},  for $n$ large
	$F^n(O)$ contains a neighborhood of a curve $C^1$-close to  $\mathbb{T} \times \{0\}$.
	In particular $F^n(O)$ intersects $W^s_{u_0} (0)$.
	One case is where $F^n(O)$ also intersects $W^s_{u_1} (1)$.
	Then further iterates $F^{n+j}(O)$ have a part that accumulates onto $\mathbb{T} \times \{1\}$ and we are done by Lemma~\ref{l:Ws-loc}.
	
	Suppose this is not the case.  Figure~\ref{f:proof-sketch} provides an illustration.
	The set $F^n(O)$ contains a small segment $[u_0 , u_0 + \varepsilon] \times \{z\}$ for some $z$ close to $0$.
	For this given $z$, Lemma~\ref{l:help} below gives the existence of a curve $a \mapsto \mathcal{C}_{u_1} (a) \in \{u_1\} \times \hat{\mathbb{C}}$ for $a \geq 0$ with $\mathcal{C}_{u_1} (0) = (u_1,0)$ and
	the following property. 
	For a compact interval $[a_1,a_2] \subset (0,\infty)$, there is $N>0$ so that $\cup_{i\le N} F^i (O)$ contains $\mathcal{C}_{u_1} ([a_1,a_2])$ in the fiber $\{u_1\} \times \hat{\mathbb{C}}$.
	Points of $\cup_{i\le N} F^i (O)$
	in $U_{u_1}$ are outside of $\overline{W^s_{u_1} (1)}$ and thus in the basin of $\infty$. Using Lemma~\ref{l:cones}, there is an iterate $F^j ( \cup_{i\le N} F^i (O))$ that  contains points outside $W^s_{u_0} (0)$. Taking $a_1$ small enough, it also contains points in $W^s_{u_0} (0)$.
	A high iterate $F^n(O)$ therefore intersects $\mathcal{J}_{u_0}$.
	
	For any open subset $W$ of $\{u_0\} \times \hat{\mathbb{C}}$ that intersects  $\mathcal{J}_{u_0}$ there is a positive integer $N$
	so that $f_{u_0}^N (W)$ covers  $\mathcal{J}_{u_0}$,
	see \cite[Corollary~11.2]{MR2193309}.
	There is $\delta>0$ so that for
	$|u-u_{0}| < \delta$, $f^N_u (W)$ also covers
	$\mathcal{J}_{u_0}$. Note that $(u_0,1) \in \mathcal{J}_{u_0}$.
	
	Combining this with the above derived statement that $ \cup_{i\le N} F^i (O)$ for $N$ high enough intersects $\mathcal{J}_{u_0}$,
	we find that $\mathbb{T} \times \{1\}$ is contained in sufficiently high iterates $F^n(O)$.
	This set $F^n(O)$ therefore intersects the basin of  $\mathbb{T} \times \{1\}$
	in a set of positive measure, by Lemma~\ref{l:Ws-loc}. So $O$  intersects the basin of  $\mathbb{T} \times \{1\}$
	in a set of positive measure, as we wanted to show.
\end{proof}	

\begin{proof}[Proof of $B^s(0) \subset B^s(\infty)$]\renewcommand{\qedsymbol}{} 		
	Consider again an open set $O$ that intersects $B^s(0)$.
	If  $F^n(O)$ also intersects $W^s_{u_1} (\infty)$, as in Figure~\ref{f:proof-sketch}, then
	further iterates $F^{n+j}(O)$ have a part that accumulates onto $\mathbb{T} \times \{\infty\}$ and we are done by Lemma~\ref{l:Ws-loc}.
	
	Now suppose that $F^n(O)$ intersects only $W^s_{u_1} (1)$ and thus does not intersect $W^s_{u_1} (\infty)$. 
	Take a periodic fiber $\tilde{u}$ in which both $0$ and $1$ are attracting.
	For $L$ odd and $\alpha$ real one can take a period-two periodic orbit $\{x_0, x_1\}$ with $x_0 =  1 / (2 (L+1))$ and $x_1= 1/2 - x_0$. Indeed, then 
	\[
	\ln|Df_{x_0} (0)| + \ln|Df_{x_1} (0)| = 
	\ln |1 + \alpha \cos(2 \pi x_0)|  + \ln|1 + \alpha\cos(2 \pi x_1)| < 0
	\] 
	and 
	\[
	\ln|Df_{x_0} (1)| + \ln|Df_{x_1} (1)| = 
	\ln |1 - \alpha \cos(2 \pi x_0)|  + \ln|1 - \alpha \cos(2 \pi x_1)| < 0.
	\]
	%
	We invoke again Lemma~\ref{l:help} below to get the following.
	A union of iterates $\cup_{i \leq N} F^i(O)$ for $N$ large enough contains  neighborhoods of $\mathcal{C}_u ([a_1,a_2])$ for  $u \in [u_0,u_1]$.
	Given $m >0$, let $[i_0,u_1]$ be an interval that maps onto $[u_0,u_1]$ injectively by $E^m$. Then $F^m ( \cup_{i \leq N} F^i(O) )$ contains  $F^m \left( \cup_{u \in [i_0,u_1]} \mathcal{C}_u ([a_1,a_2]) \right)$. 
	As $F^n(O)$ intersects $W^s_{u_1} (1)$, for suitable $a_1,a_2$ and $m$ this last set consists of curves $\mathcal{D}_u \subset \{u\} \times \hat{\mathbb{C}}$, $u \in [u_0,u_1]$, that intersect both a small neighborhood of $[u_0,u_1] \times \{0\}$ and of $[u_0,u_1] \times \{1\}$.
	The curve $\mathcal{D}_{\tilde{u}}$ therefore intersects both 
	$W^s_{\tilde{u}} (0)$ and   $W^s_{\tilde{u}} (1)$, and hence also $\mathcal{J}_{\tilde{u}}$,  
	 say in a point   $(\tilde{u} , \tilde{z})$. 
	 By Montel's theorem,  see \cite[Theorem~4.1]{MR503901}, any neighborhood of  $(\tilde{u} , \tilde{z})$ contains points in $W^s_{\tilde{u}} (\infty)$.
\end{proof}	

\noindent 	{\em Proof of $B^s(\infty) \subset B^s(0)$.}
Take an open set $O$ that intersects $B^s(\infty)$.
Recall that in a coordinate $w  = 1/z$ near $\infty$, $z_{n+1} = f_u(z_n)$ becomes
\begin{align}\label{e:z^2}
	w_{n+1}
	&=  \frac{w_n^2}{ (1 + \alpha \cos (2 \pi u)) w_n - \alpha \cos (2 \pi u)   }.
\end{align}
In the fixed fiber $u = 0$ this gives
	 \[
	 w_{n+1} = - w_n^2 / \alpha + \mathcal{O} (w_n^3).
	 \]
	 By B\"otthcher's theorem there is  a holomorphic coordinate change that conjugates this locally near $0$ to the map $q(v) \coloneqq -v^2/\alpha$. 
	 Consider a small arc $A$ close to $0$ with constant radius $\delta$. A high iterate $q^k$ maps $A$ to the entire circle $C$ with radius $1/|\alpha|^k$.	 
	 Let $J$ be a an interval in $\mathbb{T}$ containing $0$ and $1/4$. Let $I$ be the small interval in $\mathbb{T}$ containing $0$ that is mapped injectively onto $J$ by $E^k$. By Lemma~\ref{l:lin},  $\mathcal{A} \coloneqq F^k (I \times A)$ is a diffeomorphic image of the cylinder $J \times C$. 
	 By an application of Lemma~\ref{l:cones},  
	 we conclude that a high iterate $F^n(O)$ contains a small neighborhood of $\mathcal{A}$. 
	 
	  Equation~\ref{e:z^2} makes clear that
	  $
	  f_u (    ( 1 + \alpha \cos (2 \pi u))  / \alpha \cos(2 \pi u) ) = 0.
	  $
	  Near $u = 1/4$ (or $u=3/4$) the poles of $v \mapsto  1 / f_u(1/v)$ form a curve
	  in the complex plane through $0$ parameterized by $u \mapsto \alpha \cos(2 \pi u) / ( 1 + \alpha \cos (2 \pi u))$.
	  This curve intersects $\mathcal{A}$ and hence $F^n(O)$, so that there is a point $(u,z) \in F^n(O)$  with
	  $F(u,z) = (u,0)$.  There is therefore a high iterate $F^{n+m} (O)$ that contains a tubular neighborhood of $\mathbb{T} \times \{0\}$.
	 As before this implies the statement we wish to prove. 
\end{proof}

\begin{proof}[Proof of Theorem \ref{t:main}]We discuss the more involved case of real $\alpha$. At the end we discuss the changes for non-real $\alpha$. The statement that $\mathbb{T}\times\{0\}$, $\mathbb{T}\times\{1\}$ and $\mathbb{T}\times\{\infty\}$ are metric attractors follows from Lemma \ref{l:Ws-loc}.

\begin{proof}[Proof of (1)]\renewcommand{\qedsymbol}{} 	
 Iterating an open set $O\subset \mathbb{T} \times \hat{\mathbb{C}}$ by $F$ yields an open set $F^n(O)$ that projects surjectively to the first coordinate in $\mathbb{T}$. 
 As the complement $\mathcal{F}_{u_1}$ of $\mathcal{J}_{u_1}$ is open and dense in
 $\{u_0\} \times \hat{\mathbb{C}}$,  $F^n(O)$ intersects either $W^s_{u_0} (0)$ or $W^s_{u_0} (\infty)$. 
 As above, since further iterates accumulate onto $\mathbb{T} \times \{0\}$ or $\mathbb{T}\times \{\infty\}$ (see the arguments of Lemma~\ref{l:cones}),  
it follows that $O$ intersects $B^s(0)$ or $B^s (\infty)$. The statement follows from the earlier derived property that $B^s(0)$, $B^s(1)$ and $B^s(\infty)$ are all equal.  
\end{proof}

\begin{proof}[Proof of (2)]\renewcommand{\qedsymbol}{} 	
	Take open sets $U,V \subset \mathbb{T} \times \hat{\mathbb{C}}$. As $B^s(p) = \mathbb{T} \times \hat{\mathbb{C}}$ for $p=0,1,\infty$,  we find as in the reasoning  of ``$B^s(\infty) \subset B^s(0)$'', that some high iterate $F^n(U)$ contains tubular neighborhoods of $\mathbb{T} \times \{0\}$ and of $\mathbb{T} \times \{1\}$. By Montel's theorem $\cup_{n \ge 0} F^n(U)$ covers $\{u_0\} \times \mathbb{C}$. Topological mixing follows. 
\end{proof}


For non-real $\alpha$ one can follow the set-up for real $\alpha$, but some arguments become more direct and hence shorter.
	For non-real $\alpha$ we have that for any  $a \in \mathbb{C}$ we can find sequences $k,\ell$ going to infinity so that
	$\eta(u_0)^k \eta(u_1)^\ell$ converges to $a$.
	For instance in the ``proof of $B^s(0)\subset B^s(1)$'', with $O$ open and intersecting $B^s(0)$,
	this directly implies that we can find $N>0$ so that $\cup_{i\le N} F^i (O)$ has nonempty intersection with $W^s_{u_1} (1)$. So again,  further iterates $F^{n+j}(O)$ have a part that accumulates onto $\mathbb{T} \times \{1\}$ and we are done. Similar simplifications apply to the other proof parts as well.
\end{proof}

The above proof makes use of
the following lemma that adapts
results in  \cite[Lemma~3]{MR2644340} (see also  \cite[Proposition~2.1]{MR3600645}) for iterated function systems  to our skew product context.
Focus of the lemma is orbit pieces for $E$ that spend a large number, roughly $k$, of iterates near $u_0$ and then a large number, roughly $\ell$, of iterates near $u_1$.
In the lemma an asymptotic formula for the resulting orbit in the fibers is obtained.
Write $J \subset \mathbb{T}$  for one of the two intervals that form the connected components of $\mathbb{T} \setminus \{u_0,u_1\}$, and write $J = (u_0,u_1)$.

\begin{lemma}\label{l:help}
Assume the conditions of Theorem~\ref{t:main} with $\alpha$ a real number.  Fix $\hat{t} \in \overline{J}$ and  $a \in [0,\infty)$. Then there are arbitrary large integers $k, \ell$ and for each such pair of integers an orbit $t_{j} = E^j(t_0)$, $0 \leq j \leq k+ \ell$ with
$u_0 < t_0 < \cdots < t_{k} \leq u_1$ and $\hat{t} = t_{k+\ell} \leq \cdots \leq t_{k}$,
so that
\begin{align*} 
F^{k+\ell} (t_0,z) \to \mathcal{C}_{\hat{t}} (a) := \big(\hat{t} , C a z + R(z,\bar{z})\big)
\end{align*}
as $k,\ell \to \infty$.
Here $C$ is an explicit constant and $R(z,\bar{z}) = \mathcal{O} (|z|^2)$ uniformly in $a$ for $a$ from a bounded set.

For non-real $\alpha$ under the conditions of Theorem~\ref{t:main},  the same statement holds for any $a \in \mathbb{C}$.
\end{lemma}	

\begin{proof}
	We consider only real $\alpha$, the argument for complex $\alpha$ is essentially the same. So we assume eigenvalue conditions \eqref{eq:u0} and \eqref{eq:rat-indep-at0}.
	There is $i_1 \in J$ so that
	$E$ maps $[u_0,i_1]$ injectively onto $[u_0,u_1]$. Likewise there is $i_0 \in J$ so that $E$ maps $[i_0,u_1]$ injectively onto $[u_0,u_1]$.
	
	By Lemma~\ref{l:lin} we may assume the skew product is linearized near $(u_0,0)$. By iterating and using the conjugacy equation
	\[
	H_{E(u)} = \eta(u_0) H_{u} \circ f_u^{-1}
	\]
	we may assume the skew product is linearized on a neighborhood $U$ of $[u_0, u_1] \times \{0\}$. That is,  on $U$ we may use coordinates in which for points $(u,z) \in U$ so that also $F(u,z) \in U$,
	\[
	F (u,z) = (E (u) , \eta(u_0) z ).
	\]
	This works for $u \in [u_0,i_1]$ and $z$ near $0$. Note that $E(u) = u_0 + L (u-u_0)$.
	For a high natural number $k$ there is a small interval $[u_0,j_1]$ so that
	$E^k$ maps  $[u_0,j_1]$ injectively onto $[u_0, u_1]$. For $u \in [u_0,j_1]$ and $z$ near $0$,
	\[
	F^k (u,z) = (E^k (u) , \eta(u_0)^k z).
	\]
	
	In the same way there is a local linearizing coordinate  near $(u_1,0)$.
	By iterating  and using the conjugacy equation, such a coordinate exists on a small neighborhood of $[u_0,  u_1] \times \{0\}$, which we may also denote by $U$.
	That is,  on $U$ we may use coordinates in which for points $(u,z) \in U$ so that also $F(u,z) \in U$,
	\[
	F (u,z) = (E (u) , \eta(u_1) z ).
	\]
	This works for $u \in [i_0,u_1]$ and $z$ near $0$.
	For a high natural number $\ell$ there is a small interval $[j_0,u_1]$ so that
	$E^k$ maps  $[j_0,u_1]$ injectively onto $[u_0, u_1]$. For $u \in [j_0,u_1]$ and $z$ near $0$,
	\[
	F^\ell (u,z) = (E^\ell (u) , \eta(u_1)^\ell z).
	\]
	A coordinate change $(u,z) \mapsto (u,J_u(z))$ maps between the two coordinate systems.
	
	
	Consider now $t_0 \in [u_0,j_1]$ and an orbit piece
	$t_j = E^j (t_0)$, $0 \le j \le k + \ell$,  that lies inside $\overline{J}$ with $t_0 < \ldots < t_{k-1}  \in (u_0, i_1]$ and $t_k \geq \ldots \geq t_{k+\ell-1} \in [i_0,u_1]$.
	So we are considering a single orbit piece  that moves $k$ steps from $u_0$ in the direction of $u_1$ and then $\ell$ steps back.
	Note that $t_0 \to u_0$ as $k \to \infty$ and $t_k \to u_1$ as $\ell \to \infty$. We allow for the possibility $t_k = u_1$ and then $t_j = u_1$ for $k\le j \le k+\ell$.
	For orbit pieces that stay in $U$
	this gives
	\begin{align}\label{e:Fk+l}
		F^{k+\ell} (t_0,z) &=
		\left( t_{k+\ell} , J_{t_{k+\ell}}^{-1}  \big( \eta(u_1)^\ell J_{t_k} (\eta(u_0)^k z, c.c. ) \big) \right),
	\end{align}
	where $c.c.$ stands for complex conjugate.
	
	Consider such orbit pieces that stay in $U$ and end at a fixed final point $t_{k+\ell} = \hat{t} \in \overline{J}$.
	By Hypothesis~\eqref{eq:rat-indep-at0}, for any positive $a$ we can find sequences $k,\ell$ going to infinity so that
	$\eta(u_0)^k \eta(u_1)^\ell$ converges to $a$.
	If both $k,\ell \to \infty$ and  $\eta(u_0)^k \eta(u_1)^\ell$ converges to $a$,
	\[
	\eta(u_1)^\ell J_{t_k} \big(\eta(u_0)^k z, c.c.\big) \to D_zJ_{u_1} (0,0) a z
	\]
	as $t_k \to u_1$ and the higher order terms converge to zero.
	So if  $k,\ell \to \infty$ and  $\eta(u_0)^k \eta(u_1)^\ell$ converges to $a$, with $t_{k+\ell} = \hat{t}$,
	the expression \eqref{e:Fk+l} converges to
	\[
	\mathcal{C}_{\hat{t}} (a) \coloneqq 	\left( \hat{t} ,  J_{\hat{t}}^{-1}  \left( a  D_zJ_{u_1} (0,0) z \right)  \right).
	\]
	This expression makes clear that $a \mapsto \mathcal{C}_{\hat{t}} (a)$ defines a curve inside $\{\hat{t}\} \times U$, of the form
	\[
	\mathcal{C}_{\hat{t}} = \big( \hat{t} , C a z + R(z,\bar{z}) \big),
	\]
	where $C = D_zJ_{u_1} (0,0) / D_zJ_{\hat{t}} (0,0)$ and  $R(z,\bar{z}) = \mathcal{O} (|z|^2)$.
\end{proof}			
				
\section{Open problems and conjectures}\label{s:lls}

Unfortunately, we are not able to show that the union of 
$\T \times \{0\}$, $\T \times \{1\}$ and $\T \times \{\infty\}$ constitutes the likely limit set of $F$. Recall that 
a set $X$ is called the {\em likely limit set} of $F$ if $X$ is the smallest closed set that contains the $\omega$-limit sets of almost every point.

Let us explain the difficulties of proving this, by discussing a simpler case where the corresponding statement was proved. 
In \cite[Proposition 11.1]{MR2105774} it is shown that for $\alpha\in (0,1)$ real, the union of $\{0\}\times \T$ and $\{1\}\times \T$
is the likely limit set of $F\colon \T\times [0,1]\to \T \times [0,1]$, provided $L$ is sufficiently large. 

Let us sketch the proof in \cite{MR2105774}.  Let $X\subset \T \times (0,1)$ be a  forward invariant set of positive Lebesgue measure.  By Fubini’s theorem, there is exists $z \in (0,1)$ so that $X\cap (\T\times  \{z\})$ has positive measure in $\T \times (0,1)$. Take a Lebesgue density point $u\in \T$ of  $X\cap (\T\times  \{z\})$, and take a sequence of intervals $J_n \times \{z\} \subset \T\times \{z\}$ with $J_n\ni u$,  which are mapped by $F^n$ onto a curve that maps in a bijective way to $\T$ by the projection
on the first coordinate. For each $\epsilon>0$ there exists $n\ge 0$ so that the projection on the first coordinate of 
$F^n(X\cap (J_n\times \{z\}))$ contains a set of size at least $(1-\epsilon)$ in $\T$. By assumption, $F^n(J_n\times \{z\})$ intersects $\{u_0\}\times W^s(0)=\{u_0\}\times [0,1)$ 
and $\{u_1\}\times W^s(1)=\{u_1\}\times (0,1]$.  Moreover, and this sentence is crucial in the argument, 
provided we take $L$ sufficiently large, 
$F^n(J_n\times \{z\})$ is the graph of a function $\T\to \T\times (0,1)$ which 
has Lipschitz constant $\le 1$ (for any $n$), see 
Lemma~\ref{l:cones}.
By Lemma~\ref{l:Ws-loc}
 there exist $R>0$ and a set $\Lambda\subset \T$ of Lebesgue at least $1/2$ so that 
for each $u\in \Lambda$, $\{u\}\times B_R(p)\subset W^s_u(p)$ for $p\in \{0,1\}$. Combining these ingredients, it follows that for some $n'\ge n$, 
$F^n(X\cap (J_n\times \{z\}))$ must intersect either $W^s_u(0)$ or $W^s_u(1)$. Thus $X$ intersects $W^s_u(0)$ or $W^s_u(1)$
in a set of positive Lebesgue measure, and so the result follows. 

One can still find $z$ and $J_n$ as before in the complex case. Nevertheless, the previous argument does not go through even if we consider $F\colon \T\times \R \to \T \times \R$. The reason for this is that 
$F\colon \T\times \hat \C \to \T \times \hat  \C$ is discontinuous at $(u,\infty)=(1/4,\infty)$. This holds because $f_u(z)=z$ for 
$u=1/4$ and $f_{1/4+\epsilon}(z)\approx \alpha \epsilon z^2$ for $z$ large and $\epsilon\approx 0$.  
Therefore, one cannot conclude that the curve
$F^n(J_n\times \{z\})$ remains more or less horizontal, and it is entirely possible that the curve $F^n(J_n\times \{z\})$, 
considered as the graph of a function $\T\mapsto \T\times \bar \C$,  oscillates increasingly rapidly between $\T\times \{0\}$, $\T\times \{1\}$ and 
$\T\times \{\infty\}$ as $n\to \infty$.


Another approach to proving that the union of $\T\times \{0\}$, $\T\times \{1\}$ and 
$\T\times \{\infty\}$ is the likely limit set,  would be to use tools from complex dynamics. Indeed, as was done below 
Lemma~\ref{l:Ws-loc},
for each $u$ one can consider $z\mapsto f^n_u(z)$ and define the random Fatou and Julia sets $\mathcal F_u$ and  $ \mathcal J_u$. As mentioned, for certain choices of $u$,  which are periodic points of $u\mapsto E(u)$ of period $n$, 
one has that $f^n_u$ has periodic attractors other than $0,1,\infty$. So for such parameters $u$ there exist components $\mathcal F_u$ which are disjoint from the basins of $0,1,\infty$. However, it seems reasonable to 

\begin{conjecture}\label{c:1} For Lebesgue almost all $u$, 
	$\mathcal F_u$ is equal to the combined basin of $0,1,\infty$.
\end{conjecture} 
\begin{conjecture}\label{c:2} 
	For Lebesgue almost all $u$,   $\mathcal J_u$ has zero Lebesgue measure.
\end{conjecture} 
One approach to tackle Conjecture~\ref{c:2} is to consider the postcritical set $\mathcal P_u^n$ of $f^n_u$ and to take a Lebesgue density point $z$
of $\mathcal J_u$ and try to show that   for some sequence $n_i\to \infty$ the distance of $f^{n_i}_u(z)$ to $\mathcal P_u^{n_i}$ is bounded from below. If this holds, then applying Koebe and that $z$ is a Lebesgue density point, $\mathcal J_u$ would contain an open set, which would be a contradiction. 
However, it is unlikely that this approach would work for all $u$.  So probably one would need to use some parameter elimination techniques.  We do not know how to complete such a line of argument  in the current setting, although this might be easier 
to prove than the assertion that for almost every $a\in [0,4]$ the Julia set of $f_a(x)=ax(1-x)$ has zero Lebesgue measure, which follows from 
the work of  \cite{MR2018784}.

It also seems reasonable to state the following 
\begin{conjecture}\label{c:3} 
	There exists a set of positive Lebesgue measure $\Lambda\subset \T$ so that $\mathcal J_u$ is disconnected  for all $u\in \Lambda$. 
\end{conjecture}   This should hold if for almost every $u$ some critical point of $f^n_u$ escapes to infinity. There certainly exist examples for which $u$ is rational and $\mathcal J_u$ is disconnected.  For $u=1/13$,  $E^{3}(u)=u$ (when $L=3$) and the Julia set of  $f^{3}_u$ is disconnected: one of the critical points escapes. This statement does not seem to follow from previous results, where one considers
compositions of maps $f_{c_n}(z)=z^2+c_n$ where $c_n$ is chosen randomly from a large disc, see for example \cite{MR1145616,MR1721617, MR1760694,MR2218765,MR4403145}.


\appendix
	
\section{Calculation of the Lyapunov exponent}\label{s:integral}

\begin{lemma}\label{l:integral}
For $\alpha \in \mathbb{C}$ with $|\alpha|<1$,
\[
\int_0^1 \ln |1+\alpha \cos(2\pi u)| \, du  = \ln \left|  \frac{1 + \sqrt{1 - \alpha^2}}{2} \right|.
\]
\end{lemma}
	
\begin{proof}
Write $\int_0^1 \ln |1+\alpha \cos(2\pi u)| \, du =  \frac{1}{2\pi} \int_0^{2\pi} \ln |1+\alpha \cos(x)| \, dx$.
As $\ln|w| = \text{Re} (\ln(w))$, we may consider the complex function (replacing $\alpha$ by $z$)
\[
 g(z) = \frac{1}{2\pi} \int_0^{2\pi} \ln (1 + z \cos(x)) \, dx
\]
and take the real part. Note that $g(0)=0$.
For $|z|<1$, $1+z \cos(u)$ never vanishes, so we can differentiate with respect to $z$,
\begin{align*}
g'(z) &=  \frac{1}{2\pi} \int_0^{2\pi} \frac{\cos(x)}{1+z \cos(x)} \, dx
\\
&=  \frac{1}{2\pi z}  \int_0^{2\pi}  1 - \frac{1}{1+z\cos(x)} \, dx
\\
&= \frac{1}{z} - \frac{1}{2\pi z} \int_0^{2\pi}   \frac{1}{1+z\cos(x)} \, dx
\\
&= \frac{1}{z} - \frac{1}{z \sqrt{1-z^2}},
\end{align*}
using the contour integral  $\frac{1}{2\pi} \int_0^{2\pi} \frac{1}{1+z\cos(x)} \, dx =
\frac{1}{\sqrt{1-z^2}}$ (taking the principal branch of the square root).
 This yields
 \begin{align*}
 g(z) &= \int   	\frac{1}{z} - \frac{1}{z \sqrt{1-z^2}}\, dz
 \\
 &= \ln\left(   \frac{1 + \sqrt{1-z^2}}{2} \right).
 \end{align*}	
The integration constant is zero, as $g(0)=0$.	Taking the real part proves the lemma.
\end{proof}		

\begin{lemma}\label{l:alpha}
	Assume $\alpha \in \mathbb{C}$ satisfies $0 < |\alpha| < 1$ and $\left|  \frac{1 + \sqrt{1-\alpha^2}}{2} \right| < 1$. Then either $|1+\alpha| <1$ or $|1-\alpha| < 1$. 
\end{lemma}

\begin{proof}
Write $m = \min(|1-\alpha|, |1+\alpha|)$ and
assume for the sake of contradiction that $m \geq 1$. With $\alpha = x + i\,y$, this implies that both
	$|1-\alpha|^2  = (1-x)^2 + y^2  \geq 1$ and  $|1+\alpha|^2  = (1+x)^2 + y^2 \geq 1$.
So $|\alpha|^2 \geq 2x$ and $|\alpha|^2 \geq -2x$, which gives 
$|\alpha|^2 \geq 2|x|$. Squaring both sides yields
\begin{equation*}
	|\alpha|^4 \geq 4x^2.
\end{equation*}
Let $w = \sqrt{1-\alpha^2} = u + i\,v$. By the definition of the principal square root, $u \geq 0$. Since $\alpha^2 = 1 - w^2$ we have
	$x^2 - y^2 + 2i\,xy = 1 - (u^2 - v^2 + 2i\,uv)$.
Equating the real parts gives $x^2 - y^2 = 1 - u^2 + v^2$. Using the identity $2x^2 = |\alpha|^2 + \text{Re}(\alpha^2)$ yields
\begin{equation*}
	2x^2 = |\alpha|^2 + 1 - u^2 + v^2.
\end{equation*}
Substituting this into our inequality $|\alpha|^4 \geq 4x^2$ gives
$|\alpha|^4 \geq 2(|\alpha|^2 + 1 - u^2 + v^2)$,
which implies	
\begin{align*}
	2(u^2 - v^2) &\geq 2 + |\alpha|^2(2 - |\alpha|^2).
\end{align*}
Given $0 < |\alpha| < 1$, the term $|\alpha|^2(2 - |\alpha|^2)$ is strictly positive. Thus $u^2 - v^2 >1$, which implies $u>1$.
 In terms of $u$ and $v$, the condition $\frac{|1+w|}{2} = \left|  \frac{1 + \sqrt{1-\alpha^2}}{2} \right|< 1$ reads $(1+u)^2 + v^2 < 4$.
 With $u>1$ this implies $v^2 < 0$, which is not possible. 
\end{proof}	

\subsection*{Acknowledgements}

We gratefully acknowledge useful discussions with Bernold Fiedler.

\bibliographystyle{amsalpha}

\bibliography{kan}

%
%
%

\end{document}